\documentclass[11pt]{amsart}

 \usepackage{latexsym}
\usepackage{amssymb}
\usepackage{graphics}

\newtheorem{theorem}{Theorem}[section]
\newtheorem{corollary}{Corollary}[section]
\newtheorem{lemma}{Lemma}[section]
\newtheorem{conjecture}{Conjecture}[section]

\newtheorem{remark}{Remark}[section]

\newtheorem{definition}{Definition}[section]
\newtheorem{proposition}{Proposition}[section]

\def \a{\alpha }
\def \b {\beta}

\def \l{\lambda }

\newcommand{\oa} {\bar{\alpha} }

\begin{document}

\newcommand{\wta}{{\rm {wt} }  a }
\newcommand{\R}{\Bbb  R}

\newcommand{\wtb}{{\rm {wt} }  b }
\newcommand{\bea}{\begin{eqnarray}}
\newcommand{\eea}{\end{eqnarray}}
\newcommand{\be}{\begin {equation}}
\newcommand{\ee}{\end{equation}}
\newcommand{\g}{\frak g}
\newcommand{\hg}{\hat {\frak g} }
\newcommand{\hn}{\hat {\frak n} }
\newcommand{\h}{\frak h}
\newcommand{\V}{\Cal V}
\newcommand{\hh}{\hat {\frak h} }
\newcommand{\n}{\frak n}
\newcommand{\Z}{\Bbb Z}
\newcommand{\N}{{\Bbb Z} _{> 0} }
\newcommand{\Zp} {\Z _ {\ge 0} }
\newcommand{\Hp}{\bar H}
\newcommand{\C}{\Bbb C}
\newcommand{\Q}{\Bbb Q}
\newcommand{\1}{\bf 1}
\newcommand{\la}{\langle}
\newcommand{\ra}{\rangle}
\newcommand{\NS}{\bf{ns} }

\newcommand{\wt}{{\rm {wt} }   }

\newcommand{\E}{\mathcal E}
\newcommand{\F}{\mathcal F}
\newcommand{\X}{\bar X}
\newcommand{\Y}{\bar Y}

\newcommand{\hf}{\mbox{$\tfrac{1}{2}$}}
\newcommand{\thf}{\mbox{$\tfrac{3}{2}$}}

\newcommand{\W}{\mathcal{W}}
\newcommand{\non}{\nonumber}
\def \l {\lambda}
\baselineskip=14pt
\newenvironment{demo}[1]%
{\vskip-\lastskip\medskip
  \noindent
  {\em #1.}\enspace
  }%
{\qed\par\medskip
  }

\def \l {\lambda}
\def \a {\alpha}

\keywords{vertex superalgebras, affine Lie algebras, admissible representations, $N=4$ superconformal algebra, logarithmic conformal field theory}
\title[]{
  A realization of certain modules for the $N=4$ superconformal algebra and the   affine Lie algebra $A_2 ^{(1) }$}

\thanks{This work has been fully supported by Croatian Science Foundation under the project   2634 "Algebraic and combinatorial methods in vertex algebra theory"  }
\author{ Dra\v zen Adamovi\' c }

\date{}
\curraddr{Department of Mathematics, University of Zagreb,
Bijeni\v cka 30, 10 000 Zagreb, Croatia} \email {adamovic@math.hr}
\markboth{Dra\v zen Adamovi\' c} { }
\bibliographystyle{amsalpha}
\pagestyle{myheadings}

  \maketitle

\begin{abstract}
We shall first present an explicit realization of the simple  $N=4$ superconformal vertex algebra $L_{c} ^{N=4}$ with central charge $c=-9$. This vertex superalgebra is realized inside of  the $ b c  \beta \gamma $ system  and contains a  subalgebra isomorphic to  the  simple affine vertex algebra $L_{A_1} (- \tfrac{3}{2} \Lambda_0)$. Then we construct a functor from the category of $L_{c} ^{N=4}$--modules with $c=-9$ to the category of modules for the admissible affine vertex algebra $L_{A_{2} } (-\tfrac{3}{2} \Lambda_0)$. By using this construction we construct a family of weight and logarithmic modules for  $L_{c} ^{N=4}$ and  $L_{A_{2}  } (-\tfrac{3}{2} \Lambda_0)$. We also show that a coset subalgebra of $L_{A_{2}  } (-\tfrac{3}{2} \Lambda_0)$ is an logarithmic  extension of the $W(2,3)$--algebra with $c=-10$. We discuss some generalizations of our construction based on the extension of affine  vertex algebra $L_{A_1} (k \Lambda_0)$ such that $k+2 = 1/p$ and $p$ is a positive integer.
\end{abstract}

\section{Introduction}

In this paper we explicitly construct certain  simple vertex algebras associated to the $N=4$ superconformal Lie algebra and the affine Lie algebra $A_2^{(1)}$ and apply this construction  in the representation theory of  vertex algebras. We demonstrate that these vertex algebras have interesting representation theories which include finitely many irreducible modules in the category $\mathcal{O}$, infinite series of weight irreducible modules and series of logarithmic representations. We will also show that these vertex algebras are connected with logarithmic conformal field theory obtained using logarithmic extension of affine $A_1 ^{(1)}$--vertex algebras and higher rank generalizations of triplet vertex algebras.

The $N=4$ superconformal  algebra appeared in the classification of simple formal distribution Lie
superalgebras which admit a central extension containing a Virasoro subalgebra
with a non-trivial center (cf. \cite{K}, \cite{FK}). It is realized   by using quantum reduction  of affine Lie superalgebras (cf. \cite{KW2}, \cite{KWR}, \cite{Ar1}). The free-fields realization of the universal vertex algebra associated to $N=4$  superconformal algebra appeared in \cite{KW2}. In this paper we shall realize the simple $N=4$ superconformal vertex algebra $L_{c} ^{N=4}$  with central charge $c=-9$. It appears that for this central charge the simple affine vertex algebra $L_{A_1} (-\frac{3}{2} \Lambda_0)$ is conformaly embedded into vertex superalgebra $L_{c} ^{N=4}$. Moreover, the Wakimoto module for $L_{A_1} (-\frac{3}{2} \Lambda_0)$ is realized inside of vertex superalgebra $M \otimes F$, where $M$ is a Weyl vertex algebra and $F$ is a Clifford vertex superalgebra. We prove that  $L_{c} ^{N=4}$ with $c=-9$ is isomorphic to the maximal $sl_2$--integrable submodule of $M \otimes F$. In physics terminology vertex superalgebra $M \otimes F$ is called the $\beta \gamma b c$ system. So $N=4 $ superconformal vertex superalgebra is realized as a subalgebra of the $\beta \gamma b c$ system. We classify irreducible $L_{c} ^{N=4}$--modules in the category of $\tfrac{1}{2} {\Zp}$--graded modules. It turns out that our classification is similar to that of  \cite{AM}. Vertex operator superalgebra $L_{c} ^{N=4}$ has finitely many irreducible modules in the category $\mathcal{O}$ (in fact only two) and infinitely many irreducible modules in the category of weight modules. Only the category of $\tfrac{1}{2}{\Zp}$--modules with finite-dimensional weight spaces is semi-simple. By applying the construction from \cite{AdM-2009}, we construct a family of logarithmic modules for $L_{c} ^{N=4}$. Next we show that the simple affine vertex operator algebra $L_{A_{2}  } (-\tfrac{3}{2} \Lambda_0)$ can be realized as a subalgebra of  $L_{c} ^{N=4} \otimes F_{-1}$. This construction is similar to that of \cite{A-2007} where affine vertex algebra of critical level for $\widehat{sl_2}$ was realized on the tensor product of a vertex superalgebra $\mathcal{V}$ and $F_{-1}$.

As in \cite{A-2007}, we construct a family of functors $\mathcal{L}_s$ which map  (twisted) $L_{c} ^{N=4}$--modules to untwisted   $L_{A_{2}  } (-\tfrac{3}{2} \Lambda_0)$--modules. As a consequence, we construct a family of irreducible weight  $L_{A_{2}  } (-\tfrac{3}{2} \Lambda_0)$--modules and logarithmic modules.

In \cite{A-JPAA}, we presented an explicit realization of the affine vertex algebras  $L_{A_{1}  } (-\tfrac{4}{3} \Lambda_0)$ and demonstrated that this vertex operator algebra is related with triplet algebra $W(p)$ with $p=3$ (cf. \cite{AdM-08}). A connection between $L_{A_1} (-\frac{1}{2}\Lambda_0)$ and triplet algebra $W(2)$ was studied in \cite{R}. Our present construction is also related with  $\mathcal{W}$--algebras appearing in logarithmic conformal field theory. In particular, we show that the parafermion vertex  subalgebra  $K(sl_3,-\tfrac{3}{2})$ of $L_{A_{2}  } (-\tfrac{3}{2} \Lambda_0)$ is an extension of $(1,2)$--model for the $W(2,3)$-algebras. Moreover, vacuum subspace  of  $L_{A_{2}  } (-\tfrac{3}{2} \Lambda_0)$ contains vertex algebra $W_{A_2}(p)$ with $p=2$ investigated by A. M. Semikhatov in \cite{S1} (see also \cite{AdM-peking}).

For every $p \ge 3$, we also introduce  vertex algebra ${\mathcal V} ^{(p)}$ which generalize the $N=4$ superconformal vertex algebra with $c=-9$ and vertex algebra $\mathcal{R}^{(p)}$ which generalize simple affine vertex algebra $L_{A_{2}  } (-\tfrac{3}{2} \Lambda_0)$ .

In our forthcoming publications we shall study fusion rules for modules constructed in this paper.

\noindent {\bf Acknowledgment:}  This work was done in part during the author stay at Erwin Schr$\ddot{\mbox{o}}$dinger Institute in Vienna in March, 2014. We would like to thank the organizer of the program "Modern trends in topological quantum field theory" for invitation. We also thank to A. Milas, O. Per\v se, D. Ridout, T. Creutzig, A. Semikhatov and S. Woods for recent discussions on vertex algebras and logarithmic conformal field theory.

\section{Preliminaries}
\label{prel}
  In this section we recall the
definition of vertex   superalgebras, their twisted modules (cf.
\cite{FHL}, \cite{FLM}, \cite{K},  \cite{LL}).

Let $(V=  V _{\bar 0} \oplus V _{\bar 1},Y, {\bf 1}, \omega)$ be a vertex operator superalgebra. We
shall always assume that
\bea
&&V_{\bar 0}=\coprod_{ n \in {\Zp} } V(n), \quad V _{\bar 1} = \coprod_{ n \in \tfrac{1}{2} +{\Zp} } V(n) \nonumber \\
&&\mbox{where} \ \  V(n) = \{ a \in V \ \vert \ L(0) a = n v \}. \nonumber \eea
For $a \in V(n)$, we shall write $\wt (a) = n$, or ${\rm deg}(a)=n$.
As usual, vertex operator associated to $a \in V$ is denoted by $Y(a,x)$, with the mode expansion
$$Y(a,x)=\sum_{n \in \mathbb{Z}} a_n x^{-n-1}.$$

Any element $u \in V _{\bar 0} $ (resp. $u \in
V _{\bar 1})$ is said to be even (resp. odd). We define $\vert u
\vert = 0 $ if $u$ is even and $\vert u \vert = 1$ if $u$ is odd.
Elements in  $V _{\bar 0}$ or  $V _{\bar 1}$ are called
homogeneous. Whenever $\vert u \vert $ is written, it is
understood that $u$ is homogeneous.

Let $\sigma$ be canonical automorphism of $V$ of order two.

Assume now that $g$ is an automorphism of the vertex superalgebra
$V$ such that $ g $ acts semisimply on $V$ and
\bea \label{svojstvo-automorfizma}
&&   V=\oplus _{ \bar{\a} \in \Gamma/{\Z} }   V ^{\bar{\a}}  \label{aut-1} \\
 && g  v = e^{2 \pi i {\a}} v, \ \ \mbox{for} \ \ v \in V^{\bar{\a}} \label{aut-2}
\eea
where $\Gamma$ is an additive subgroup of ${\R}$ containing ${\Z}$ and ${\oa}= {\a} + {\Z}$. We will always assume that $0 \le \alpha < 1$.

\begin{definition}
Let  $g$ be an automorphism of $V$  be such that
(\ref{aut-1}) and (\ref{aut-2})  hold. A  weak $g$--twisted $V$--module is a pair
$(M, Y_M)$, where $M = M _{\bar 0} \oplus M _{\bar 1}$ is a
${\Z}_2$--graded vector space, and $Y_M (\cdot, z)$  is a linear
map
$$ Y_M : V \rightarrow \mbox{End} (M)\{ z \}, \ a \mapsto
Y_M(a,z) = \sum_{ n \in \Gamma} a_n z ^{-n-1}, $$
satisfying the following conditions for $a, b \in V$ and $v \in
M$:

\begin{enumerate}

\item [(M1)]  $\vert a_n v\vert$ = $\vert a\vert + \vert v\vert$
for any $a \in V$.
 \item [(M2)]  $Y_M({\bf 1},z)=I_M$.
\item[ (M3)]$Y_M(D a,z)=\frac{d}{dz}Y_M(a,z)$.
\item[(M4)] $a_n v = 0$ for $n \gg 0$.
\item[(M5)] $Y_M(a,z) = \sum_{ n \in {\oa} } a_n z^{-n-1}$  for $a
\in  V ^{\bar{\a}} $.

\item[(M6)] The twisted Jacobi identity holds

\begin{eqnarray}
&
&z_{0}^{-1}\delta\left(\frac{z_{1}-z_{2}}{z_{0}}\right)Y_M(a,z_{1})Y_M(b,z_{2})
 -(-1) ^{\vert a \vert \vert b \vert}
z_{0}^{-1}\delta\left(\frac{z_{2}-z_{1}}{-z_{0}}\right)
Y_M(b,z_{2})Y_M(a,z_{1})\nonumber \\
&=&z_{2}^{-1} \left(\frac{z_{1}-z_{0}}{z_{2}}\right) ^{-\a}
\delta\left(\frac{z_{1}-z_{0}}{z_{2}}\right)Y_M(Y(a,z_{0})b,z_{2}).
\nonumber
\end{eqnarray}
for $ a \in V^{\bar{\a}}$, $b \in V$.

\end{enumerate}
\end{definition}

\begin{definition}
An admissible $g$--twisted $V$--module is a weak $g$--twisted $V$--module $M$ with a grading of the form
$$ M = M(0) \bigoplus \bigoplus_{ \mu  > 0 } M(\mu) $$
such that
$M(0) \ne 0$ and for any homogeneous  $a \in V$, $n \in {\Gamma}$ we have
$$ a_n  M(\mu) \subset M( \mu + \wt  (a)  - n  -1). $$
\end{definition}

In the case $g = 1$, $M$ is called (untwisted) $V$--module.

 Now we shall recall one important construction of twisted modules.
Let now $ h \in V$ such that
$$L(n) h =\delta_{n,0} h, \ \ h(n) h = \delta_{n,1} \gamma {\1} $$
for any $n \in {\Zp}$, where $\gamma$ is a fixed complex number.
Assume that $h(0)$ acts semisimply on $V$. For simplicity we assume that $h(0)$ has real eigenvalues. Let $\Gamma$ be an
additive subgroup of ${\R}$ generated by $1$  and the eigenvalues
of $h(0)$. Then $g_h= e^{2 \pi i h(0)}$ is an automorphism of $V$
and (\ref{aut-1})-(\ref{aut-2})  hold. $g_h$ has finite order if and only if the action of $h(0)$ on $V$ has only rational eigenvalues.

\begin{remark}
In general the automorphism $g$ can have infinite order, and group
${\Gamma}/ {\Z}$ is infinite. Using slightly different
terminology, $g$--twisted modules for ${\Gamma}/{\Z}$--graded
vertex superalgebras  were defined and investigated in \cite{KR}.
\end{remark}

Define $$\Delta (h, z ) = z^{ h(0)} \exp \left(\sum_{ n=1}
^{\infty} \frac{h(n)}{-n} (-z)^{-n}\right).$$
Applying  the results obtained in  \cite{Li5} on
$V$--modules we get the following proposition.

\begin{proposition} \label{novi-moduli}
%
For any weak $V$--module $(M,Y_M(\cdot,
 z))$,
$$ (\widetilde{M}, \widetilde{Y}_{\widetilde{M} } (\cdot,z)):=(M, Y_M (\Delta (h, z) \cdot, z))$$
is a $g_h$--twisted weak $V$--module. $\widetilde{M}$ is an irreducible
twisted $V$--module if and only if
  $M$ is   an irreducible  $V$--module.
\end{proposition}

Let us recall the definition of Zhu's algebra for vertex operator
superalgebras.

 Let $g$ be an automorphism of the vertex operator superalgebra $V$ such that (\ref{aut-1})- (\ref{aut-2}) hold. Assume also  that  for the automorphism  $  g \sigma  $ we have
\bea \label{svojstvo-automorfizma-2}
&&   V=\oplus _{ \bar{\a} \in \widetilde{\Gamma}/{\Z} }   V ^{\bar{\a}}  \label{aut-3} \\
 &&  g \sigma  v = e^{2 \pi i {\a}} v, \ \ \mbox{for} \ \ v \in V^{\bar{\a}} \label{aut-4}
\eea
where $\widetilde{\Gamma}$ is an additive subgroup of ${\R}$ containing ${\Z}$ and ${\oa}= {\a} + {\Z}$ where $0 \le \alpha < 1$.

Now we recall definition of twisted Zhu's algebra for vertex operator superalgebras (for various versions of the definition  see   \cite{Z}, \cite{Xu}, \cite{DZ}, \cite{DLM}, \cite{KS}, \cite{Ek} ).
For homogeneous $v \in V^{\bar \alpha}$ we set $\delta_{\bar \alpha} = 1$ if $ \bar{\alpha} = \bar{0} $ and $ \delta_{\bar \alpha} = 0$ otherwise.
We define two bilinear maps $*_g : V  \times V \rightarrow V$,
$\circ_g : V \times V \rightarrow V$ as follows: for homogeneous $a,
b \in V$, let
\bea
a *_g b &=& \left\{\begin{array}{cc}
 \  \mbox{Res}_x Y(a,x) \frac{(1+x) ^{\wt (a)}}{x}b  & \mbox{if} \ a,b  \in V^{\bar{0}}  \\
  0 & \mbox{if} \ a \ \mbox{or} \ b  \in V^{\bar{\alpha}} ,  \bar{\alpha} \ne \bar{0} \
\end{array}
\right.  \\
a\circ_g  b  &=&
   \mbox{Res}_x Y(a,x) \frac{(1+x) ^{\wt  (a) -1 + \delta_{\bar \alpha} + \alpha } }{x^{1+ \delta_{\bar \alpha} } }b  \quad  \mbox{if} \ a  \in V^{{\bar \alpha }}
 \eea

Next, we extend $*_g$ and $\circ_g$ to $V \otimes V$ linearly, and
denote by $O_g(V)\subset V$ the linear span of elements of the form
$a \circ_g b$, and by $A_g(V)$ the quotient space $V / O_g(V)$. The
space $A_g(V)$ has a unitary associative algebra structure, with the multiplication induced by $*_g$. Algebra $A_g(V)$ is called  the $g$--twisted
Zhu algebra of $V$.    The image of $v \in V$, under the natural
map $V \mapsto A_g(V)$ will be denoted by $[v]$.

If $g=Id$ we shall denote $A_g(V)$ by $A(V)$.

The following theorem was proved in \cite{DZ} (see also \cite{Z}, \cite{DLM}, \cite{Ek},  \cite{KS}).

\begin{theorem} \label{zhu-corr-mod} There is a one-to-one correspondence between admissible irreducible $g$--twisted  $V$-modules and
irreducible $A_g(V)$--modules.
\end{theorem}

\section{ Vertex operator algebra $L (k\Lambda_0)$ }

In this section we recall some basic facts about vertex operator
algebras associated to affine Lie algebras (cf. \cite{FZ},
\cite{Li-local}, \cite{K}).

Let ${\g}$ be a finite-dimensional simple Lie algebra over ${\Bbb
C}$ and let $(\cdot,\cdot)$ be a nondegenerate symmetric bilinear
form on ${\g}$. Let  ${\g} = {\n}_- + {\h} + {\n}_+$    be a
triangular decomposition for ${\g}$.
 The affine Lie algebra ${\hg}$ associated
with ${\g}$ is defined as $ {\g} \otimes {\C}[t,t^{-1}] \oplus
{\C}c \oplus {\C}d, $ where $c$ is the canonical central element
\cite{K}
 and  the Lie algebra structure
is given by $$ [ x \otimes t^n, y \otimes t^m] = [x,y] \otimes t
^{n+m} + n (x,y) \delta_{n+m,0} c,$$ $$[d, x \otimes t^n] = n x
\otimes t^n $$ for $x,y \in {\g}$.  We will write $x(n)$ for $x
\otimes t^{n}$.

The Cartan subalgebra ${\hh}$ and  subalgebras ${\hg}_+$,
${\hg}_-$ of ${\hg}$ are defined by $${\hh} = {\h} \oplus {\C}c
\oplus {\C}d, \quad
 {\hg}_{\pm} =  {\g}\otimes t^{\pm1} {\C}[t^{\pm1}].$$

Let  $P = {\g}\otimes {\C}[t] \oplus {\C}c \oplus {\C}d$ be upper
parabolic subalgebra.   For every $k \in {\C}$,  let ${\C} v_k$ be
$1$--dimensional $P$--module  such that the subalgebra ${\g}
\otimes  {\C}[t] + \C d$ acts trivially,
 and  the central element
$c$ acts as multiplication with $k \in {\C}$. Define the
generalized Verma module $N_{\g} ( k \Lambda_0)$ as
$$N_{\g} ({k}\Lambda_0) = U(\hg) \otimes _{ U(P) } {\C} v_k .$$
Then $ N_{\g} (k \Lambda_0)$ has a natural structure of a vertex
algebra generated by fields
$$ x (z) = Y( x(-1) {\bf 1}, z) = \sum_{n \in {\Z} } x(n) z ^{-n-1} \quad (x \in {\g} ),$$
where  ${\1} = 1 \otimes v_k$ is the vacuum vector.

Let $N^{1} _{\g} (k \Lambda_0)$ be the maximal ideal in the vertex
operator algebra $N_{\g} (k \Lambda_0)$. Then $L_{\g} (k \Lambda_0) =
\frac{N_{\g}(k \Lambda_0)}{N^{1}_{\g} ( k \Lambda_0)}$ is a simple vertex
operator algebra.

If $\g$ is a simple Lie algebra of type $A_n$, we shall denote the above vertex operator algebras  by $N_{A_n} (k \Lambda_0)$ and $L_{A_n} (k \Lambda_0)$.

Let $M_{\h} (k)$ denotes the vertex subalgebra of $L_{\g} (k \Lambda_0)$ generated by fields $h(z)$, $h \in {\h}$. Recall that if $k \ne 0$ then $M_{\h}(k) \cong M_{\h} (1)$ (cf. \cite{LL}).  We have the following coset vertex algebra
$$ K(\g, k):=\mbox{Com} (M_{\h} (k), L(k \Lambda_0) ) = \{ v \in L(k \Lambda_0) \ \vert \ h(n) v = 0, \ h \in {\h}, n \ge 0 \}, $$
which is called the parafermion vertex algebra.

The structure and representation theory   of the  parafermion vertex algebras  in the case when $k \in {\Zp}$ have been developed in series of papers \cite{ALY}, \cite{DLWY},  \cite{DW}. We shall see that for admissible affine vertex algebras the parafermion vertex  algebras are related with vertex algebras appearing in logarithmic conformal field theory.

\section{Lattice construction of $A_1^{(1)}$--modules of level
$-\frac{3}{2}$ and screening operators}
\label{real-sl2}

In this section we shall recall the definition of the Wakimoto modules for affine Lie algebra $A_1^{(1)}$ in the case of level $k = - 3/2$. For this level we will have  nice realization of modules and screening operators using lattice vertex superalgebras and their twisted modules.
Details about Wakimoto modules can be found in \cite{efren}, \cite{FB} and \cite{W-mod}.

Let $ p\in {\N}$, $p \ge 2$.

Let $L^{(p)} $ be the following lattice
$$L^{(p)}= {\Z} \a + {\Z}\b + {\Z} \delta $$
with the ${\Q}$--valued bilinear form $\la \cdot, \cdot \ra$ such
that
$$ \la \a , \a \ra = 1,  \quad \la \b , \b
\ra =-1, \quad \la \delta, \delta \ra = \frac{2}{p} $$ and other products of basis vectors are zero.

Let $V_{ L ^{(p)} } $ be the associated  generalized vertex algebra (cf. \cite{DL}). If $p=2$  then $V_{ L ^{(2)} } $ is a vertex superalgebra.

Assume now that $k + 2 = \frac{1}{p}$.

Define the
following three vectors
\bea
&& e = e^{\a + \b}, \label{def-e}  \\
&& h = -2 \b (-1) + \delta(-1), \label{def-h} \\
&& f = ( (k+1)  ( \a (-1) ^{ 2} - \a(-2) ) - \a(-1)
\delta(-1) +  \nonumber \\ && \ \ \  (k+2)  \a(-1) \b (-1)  )   e^{-\a - \b}.
\label{def-f} \eea
Then the components of the fields $Y(x,z) = \sum_{n \in {\Z}} x(n)
z ^{-n-1}$, $x \in \{e,f,h\}$ satisfy the commutations relations
for the affine Lie algebra $\hat{sl_2}$ of level $k$. Moreover, the subalgebra of $V_{ L ^{(p)} }$ generated by the
set $\{e,f,h\}$ is isomorphic to the simple vertex operator
algebra $L_{A_1} ( k \Lambda_0)$.

Screening operators are
\bea
\label{scr-gen}   Q = \mbox{Res}_z Y( e^{\alpha + \beta - p  \delta}, z) , \quad  \widetilde{Q} = \mbox{Res}_z Y( e^{ - \tfrac{1}{p} (\alpha + \beta)  +\delta}, z).
\eea
They commute with the $\widehat{sl_2}$--action.

Let $M$ be a subalgebra of $V_{L^{(p)} }$ generated by
$$ a  = e ^{\alpha + \beta}, \ a ^{*} =-\alpha(-1) e^{-\alpha -
\beta}. $$ Then $M$ is isomorphic to the Weyl vertex algebra (cf.
\cite{FMS}, \cite{efren}, \cite{A-2001}).

Let $M_{\delta}(1)$ be the Heisenberg vertex algebra generated by the field $\delta (z) = \sum_{n \in {\Z}} \delta(n) z ^{-n-1} $. Then
$$ F_{p/2} = M_{\delta} (1) \otimes {\C}[ {\Z} \tfrac{p}{2} \delta]  \quad \mbox{and} \quad  M \otimes F_{p/2} $$ are subalgebras of $V_{L^{(p)} }$.

We have the following (generalized) vertex algebra

$$ {\mathcal V} ^{(p)} = \mbox{Ker}  _{ M \otimes F_{p/2} } \widetilde{Q}. $$

Moreover  $L(k \Lambda_0)$ can be realized as a subalgebra of the
 $M \otimes M_{\delta}(1) \subset  M \otimes F_{p/2} $ (cf. \cite{efren},
\cite{FB}, \cite{W-mod}):
\bea
&& e(z)  = a (z) , \label{def-e-cw} \\
&& h(z) = -2 : a ^{*} (z) a (z) : +  \delta(z)  , \label{def-h-cw} \\
&& f(z)  = - : a^{*} (z) ^{2} a (z) : +  k \partial_{z}
a^{*} (z) +    a^{*} (z) \delta(z) . \label{def-f-cw} \eea
Since $ \widetilde{Q}$ commutes with the action of  $\widehat{sl_2}$ we have that  $$L_{A_1} (k \Lambda_0)  \subset {\mathcal V} ^{(p)}. $$ Moreover, one can show that $Q$ acts as a derivation on ${\mathcal V} ^{(p)}$.
If $p$ is even, then ${\mathcal V} ^{(p)}$ is a vertex superalgebra.

Let now $p = 2$ and $k =-3/2$. Let $L= L^{(2)}$. Then $V_L$ is a vertex operator superalgebra and the above formula give a realization of the simple vertex operator algebra inside of $V_L$.

Under this realization, the Sugawara vector is mapped to
\bea && \omega = \frac{1}{2} ( \a(-1) ^{2} -\a(-2) -\b(-1) ^{2} +
\b (-2) + \delta(-1) ^{2} - 2 \delta(-2) ). \label{def-virasoro}
\eea
The components of the field $Y(\omega,z) = \sum_{n \in {\Z}} L(n)
z^{-n-2}$ satisfy the commutation relations for the Virasoro
algebra with central charge $c=-9$.

Screening operators are  $$ Q= e^{\a + \b - 2 \delta}_0 = \mbox{Res}_z
Y(e^{\a+ \b -2\delta},z), \quad \widetilde{Q} = e^{-\frac{\a + \b
- 2 \delta} {2} }_0 = \mbox{Res}_z Y(e^{-\frac{\a + \b - 2 \delta}
{2} },z).  $$

 We also note the following important relation:

\bea
&& e(-1) \omega = e^{\a + \b }_{-1} \omega = \nonumber \\
 &&\left( - (\a(-1) + \b(-1) ) \b(-1) + \b (-2) +
\tfrac{1}{2}\delta(-1) ^{2} - \delta(-2) \right) e^{\a + \b}.
\label{vazna-rel1}
\eea

Moreover,  $F= F_1$ is a Clifford (fermionic) vertex superalgebra and it is
generated by $ \Psi = e ^{\delta}, \ \Psi ^{*} = e ^{-\delta}$
(cf. \cite{K}).

 Then  formulas (\ref{def-e-cw})-(\ref{def-f-cw}) give a realization of   $L(-\frac{3}{2} \Lambda_0)$  as a vertex subalgebra of the
Clifford-Weyl vertex superalgebra $M \otimes F$.
%
%
%
%

Since $e ^{\alpha + \beta - 2 \delta } =a _ {-1} \Psi^{*}_{-2}
{\Psi} ^{*} \in M \otimes F$ we have that
$$ Q = \mbox{Res}_{z} : a(z) \ \partial_z\Psi^{*} (z) \ \Psi^{*} (z):
. $$

Since $sl_2$ acts on $M \otimes F$ by derivations, we have the
subalgebra $(M \otimes F  )^{int}$ which is the maximal $sl_2$--integrable
submodule of $M \otimes F$.

Clearly, $L_{A_1} (-\frac{3}{2} \Lambda_0) \subset (M \otimes F  )^{int}$. Since
$e^{\delta}$ is $\widehat{sl_2}$--singular vector and since
\bea \label{integ}  f(0) ^2 e^{\delta}  =0  \eea
we have that $\Psi = e ^{\delta} \in   (M \otimes F  )^{int}$.

\begin{remark}
In the physics literature  the tensor product vertex algebra $M \otimes F$ is called the $\beta \gamma b c$  system. So we have realized an  extension of $L_{A_1} (-\tfrac{3}{2} \Lambda_0) $ inside of $\beta \gamma b c $ system.
 A different kind of realization of  $L_{A_1} (-\tfrac{3}{2} \Lambda_0) $ also appeared in \cite{CL}. The authors realized  $L_{A_1} (-\tfrac{3}{2} \Lambda_0) $ inside of the tensor product of three copies of the   $\beta \gamma b c$  system.
It was shown in \cite{CL} that the simple affine vertex algebra $L_{A_1} (-\tfrac{3}{2} \Lambda_0) $ (denoted there by $V_{-3/2} (sl_2)$) is connected with Odake's algebra with central charge $c=9$. In fact Odake's algebra and $V_{-3/2} (sl_2)$ were used to describe certain coset vertex subalgebras of  $(M \otimes F)^{\otimes 3}$. Recall that   Odake's vertex algebra is an extension of $N=2$ superconformal vertex algebra with central charge $c=9$. In the present  paper we shall go in a different direction. We shall study $(M \otimes F  )^{int}$ and prove that it is isomorphic to the simple $N=4$ superconformal vertex algebra with central charge $c=-9$, which is a different extension  of the $N=2$ superconformal vertex algebra.
\end{remark}

Let now $\mu \in {\C}$ and let $g_\mu:= e^{ 2 \pi i \mu \delta (0) }$. Then $g_{\mu}$--is an automorphism of $F$. Let $F ^{\mu}$ be a $g_{\mu}$--twisted $F$--module:
$$ (F^{\mu}, \widetilde{Y}_{F ^{\mu} } (\cdot, z) ):= (F, Y_{F} (\Delta (\mu \delta (-1) {\bf 1}) \cdot, z) ). $$

\begin{remark}
Recall that the twisted module structure on $F^{\mu}$  is generated by the twisted fields:
$ z ^{\mu} \Psi(z), z ^{-\mu} \Psi^* (z)$
(see \cite{KR}).
\end{remark}

$F^{\mu} $ is untwisted  $M_{\delta} (1)$--module and we have the following decomposition
$$ F ^{\mu} \cong V_{\Z \delta} . e ^{\mu \delta} = \bigoplus_{ j \in {\Z} } M_{\delta} (1). e ^{(j+ \mu) \delta}. $$

Identifying  $g_{\mu} =  Id \otimes  g_{\mu} $, we can consider $g_{\mu}$ as an automorphism of the vertex operator superalgebra $M \otimes F$. Moreover, if $\mathcal{M}$ is any module for the vertex algebra $M$, then $\mathcal{M} \otimes F ^{\mu}$ is a $g_{\mu}$--twisted $M \otimes F$--module.  Since $L_{A_1} (-\tfrac{3}{2} \Lambda_0) \subset M \otimes M_{\delta} (1)$ we have that  $\mathcal{M} \otimes F ^{\mu}$  is an untwisted $L(-\tfrac{3}{2} \Lambda_0)$--module.
 In particular,
 $$ M \otimes F ^{\mu} = \bigoplus_{j \in {\Z} } M \otimes M_{\delta}(1). e ^{ (\mu + j) \delta }, $$
 so $M \otimes F ^{\mu}$ is a direct sum of Wakimoto modules for $\widehat{sl_2}$ of the form $M \otimes M_{\delta} (1). e ^{ (\mu + j) \delta }$.

\section{ A realization of the simple $N=4$ superconformal vertex algebra with central
charge $c=-9$}

In this section we shall present an explicit realization of the simple $N=4$ superconformal vertex algebra $L^{N=4}_c$  with central charge $c=-9$.

We shall first recall the definition of the universal $N=4$ superconformal vertex algebra $V^{N=4}_c$ of central charge $c$ associated to the $N=4$ superconformal algebra.
Given a vertex superalgebra $V$ and a vector $a \in V$, we expand
the field $Y(a,z) = \sum_{n \in {\Z}} a_n z^{-n-1}$.  Define the
$\lambda$-bracket
$$[a_{\lambda} b] = \sum_{j \ge 0} \frac{\lambda ^{j}}{j!} a_j b.
$$
We shall now define the  "small" $N=4$ superconformal vertex
algebra $V_c ^{N=4}$ following \cite{K}. The even part of this
vertex algebra is generated by the Virasoro field $L$ and three
primary fields of conformal weight $1$, $J^{0}$, $J^{+}$ and
$J^{-}$. The odd part is generated by four primary fields of
conformal weight $\tfrac{3}{2}$, $G^{\pm}$ and
$\overline{G}^{\pm}$.

The remaining (non-vanishing) $\lambda$--brackets are
\bea
 &[J^{0}_{\l}, J^{\pm}] = \pm 2 J^{\pm} & [J^{0}_{\l} J^{0}] =
 \frac{c}{3} \l \nonumber \\
 &[J^{+}_{\l} J^{-}] = J^{0} + \frac{c}{6} \l  & [J^{0} _{\l}
 G^{\pm}] = \pm G^{\pm} \nonumber \\
 &   [J^{0}_{\l} \overline{G} ^{\pm} ] = \pm \overline{G} ^{\pm} &
 [J ^{+}_{\l} G^{-} ] = G^{+} \nonumber \\
 &   [J^{-}_{\l} G ^{+}] = G ^{-}  & [J^{+}_{\l}
 \overline{G}^{ -}] = - \overline{G} ^{+} \nonumber \\
 & [ J^{-}_{\l} \overline{G} ^{+}] = - \overline{G} ^{ -} &
 [G^{\pm}_{\l} \overline{G} ^{\pm} ] = (T + 2 \l) J^{\pm}
 \nonumber \\
 & [G^{\pm}_{\l}  \overline{G} ^{\mp} ] =& L \pm \frac{1}{2} T J^{0} \pm \l J^{0} + \frac{c}{6} \l ^{2}
 \nonumber
\eea

Let $V^{N=4}_c$ be the vertex superalgebra freely generated by the
fields $\overline{G} ^{\pm}$, $G^{\pm}$, $J^{+}$, $J^{0}$, $J^{-}$
and $L$. Let $L^{N=4}_c$ be its simple quotient.

Irreducible highest weight $V^{N=4}_c$--modules are defined as usual (cf. \cite{KW2} and \cite{Ar1}).

We shall now show that the simple vertex superalgebra $L^{N=4}_c$
with $c=-9$ can be realized as a subalgebra of the lattice vertex
superalgebra $V_L$ from  Section \ref{real-sl2}.

Define

\bea
&& j^{+} =e,  \quad j^{0}=h, \quad j^{-}= f, \label{def-j} \\
&& \tau^{+} =   e^{\delta} , \label{def-t+} \\
&& \overline{\tau} ^{+} = Q e^{\delta} = (\a(-1) + \b(-1) -2
\delta(-1)) e^{\a+\b -\delta}, \label{def-ot+} \\
&& \tau^{-}= f(0) e^{\delta} =-\alpha (-1) e ^{- \alpha - \beta + \delta}, \label{def-t-} \\
&& \overline{\tau} ^{ -}= - f(0) Q e^{\delta}. \label{def-ot-} \\
&&  \omega = ( e(-1) f(-1) + f(-1) e(-1) + \frac{1}{2} h(-1) ^2 ) {\bf 1}  \nonumber \\ && = \quad \frac{1}{2} ( \a(-1) ^{2} -\a(-2) -\b(-1) ^{2} +
\b (-2) + \delta(-1) ^{2} - 2 \delta(-2) ). \label{def-virasoro-2}
\eea

(Note  that the Virasoro vector $\omega$ is the Sugawara Virasoro vector in $L_{A_1} (-\frac{3}{2} \Lambda_0)$).

We can evaluate  the above expressions in vertex superalgebra $M \otimes F$.

 $$  \tau ^{+} = \Psi, \quad \tau^{-} = a^{*} _{-1}  \Psi, $$
 $$ \overline{\tau} ^{+}  = 2 a_{-1}  D \Psi^{*} + a_{-2} \Psi^*, $$
 $$ \overline{\tau} \  ^{-}  = D^{2} \Psi^{*} + (2 a^{*}_{-1} a_{-1} + \delta (-1) ) D \Psi^{*} + a^{*}_{-1} a_{-2}
  \Psi^{*}. $$

Define also the following fields

 \bea
&& J^{+} (z) =e(z),  \quad J^{0} (z)=h(z) , \quad J^{-}= f(z), \label{def-J(z)} \\
&&G^{\pm} (z) = Y( \tau^{\pm},z) = \sum_{n \in {\Z}} G^{\pm}(n+\tfrac{1}{2}) z^{-n-2}  , \label{def-G(z)} \\
&& \overline{G} ^{\pm} (z) = Y(\overline{\tau}^{\pm},z)  =
\sum_{n \in {\Z}} \overline{G}^{\pm}(n+\tfrac{1}{2}) z^{-n-2}, \label{def-oG(z)} \\
&& L(z) = Y(\omega,z) = \sum_{\n \in {\Z}} L(n) z^{-n-2}.
\label{def-L(z)}
\eea

\begin{lemma}
The components of the fields (\ref{def-J(z)})-
(\ref{def-L(z)}) satisfy the commutation relations  of the $N=4$
superconformal algebra with central charge $c=-9$.
\end{lemma}
\begin{proof}
Standard calculation in lattice vertex algebras implies
$$ \tau^{+}_0  \overline{\tau} ^{\pm}  =J^{\pm}(-2){\bf 1}, \quad  \tau^{+}_1  \overline{\tau} ^{\pm}= 2 J^{\pm}(-1){\bf 1}, \quad \tau^{\pm}_n  \overline{\tau} ^{\pm}  = 0
\quad (n \ge 2).$$
 Moreover  \bea && \tau^{+}, \overline{\tau}^{-}, j^{0}, \omega \label{n=2-1} \\  &&\tau^{-}, \overline{\tau}^{+}, -j^{0}, \omega \label{n=2-2} \eea
  are $N=2$ superconformal vectors, i.e.,   they generate  free-field realizations of $N=2$ superconformal algebra with $c=-9$ inside of $\beta \gamma b c $ system.
By using relation (\ref{integ}) and  the fact that $Q$ is a screening operator we  have that $\tau^{+}, \overline{\tau} ^{+}$ are highest weight vector for the affine Lie algebra $\widehat{sl_2}$ generated by $J^{\pm}(z)$ and $J^0(z)$ and
$$ U(\widehat{sl_2}). \tau^{+} \cong U(\widehat{sl_2}). \tau^{+} \cong L_{A_1} (-\frac{5}{2} \Lambda_0 + \Lambda_1) $$
In particular $U(sl_2). {\tau}^{+}$  (resp. $U(sl_2). \overline{\tau}^{+}$ ) is two-dimensional  spanned by $\tau^{\pm}$ (resp. $\overline{\tau}^{\pm}$).
From   the arguments described above and using commutator formula, we get the assertion.
\end{proof}
 Let us denote the $N=4$ superconformal algebra with ${\mathcal
A}$.
Consider
$$ V= U({\mathcal A}).{\1}. $$
Then $V$ is a highest weight ${\mathcal A}$--module. Moreover, $V$
is a $\tfrac{1}{2} {\Zp}$--graded vertex operator superalgebra
with central charge $c=-9$.

So we have:

\begin{proposition}
$V$ is a subalgebra of the Clifford-Weyl vertex superalgebra $M
\otimes F$.
\end{proposition}

\begin{lemma} \label{lemma-c}We have:
$$ G^{+}(-\tfrac{3}{2}) \overline{G}^{+} (-\tfrac{3}{2}) {\1} = -2
e(-1) \omega + h(-1) e(-2) - h(-2) e(-1). $$
\end{lemma}
\begin{proof} By using our realization and standard calculation in
lattice vertex superalgebras we get
\bea && G^{+}(-\tfrac{3}{2}) \overline{G}^{+} (-\tfrac{3}{2}) {\1}
= e^{\delta}_{-1}(\a(-1) + \b(-1) -2 \delta(-1)) e^{\a+\b -\delta}
\nonumber \\
&&= \left( (\a(-1) + \b(-1) ) \delta(-1) - \delta(-1) ^{2} +
\delta(-2) \right) e^{\a + \b}. \label{rel-pom2} \eea
By using (\ref{vazna-rel1}) we get
\bea
&&G^{+}(-\tfrac{3}{2}) \overline{G}^{+} (-\tfrac{3}{2}) {\1} + 2
e(-1) \omega \nonumber \\
&& = \left( (\a(-1) + \b(-1) ) (- 2 \b(-1) + \delta (-1)) \right)
e^{\a + \b} + \left( 2 \beta(-2) - \delta(-2) \right) e^{\a + \b}
\nonumber \\
&& = h (-1) e(-2) - h(-2) e(-1), \nonumber \eea
and the lemma follows. \end{proof}

\section{Irreducible $V$--modules in the category $\mathcal{O}$ }

In this section we shall classify irreducible $V$--modules which are in the category $\mathcal{O}$ as modules for the $N=4$ superconformal algebra. This classification reduces to the classification of all irreducible highest weight modules for the Zhu's algebra.

So we shall first identify Zhu's algebra of $V$.
\begin{proposition} \label{zhu-str}
\item[(i)] Zhu's algebra $A(V)$ is isomorphic to a certain
quotient of $U(sl_2)$.

\item[(ii)] In Zhu's algebra $A(V)$ we have the following
relation:
$$ [e] ( [\omega] + \tfrac{1}{2}) = 0.$$
\end{proposition}
\begin{proof}  Since $V$ is strongly generated by $ \tau^{\pm}, \overline{\tau} ^{\pm}, e, h, f $ we have that Zhu's algebra $A(V)$ is generated by
 $$ [\tau ^{\pm} ], [{ \overline{\tau} } ^{\pm}], [e], [f], [h]. $$
 From the  definition of Zhu's
algebra for vertex operator superalgebras follows that
 $$ \tau ^{\pm}, \overline{\tau}^{\pm} \in O(V)$$ and therefore
 $[\tau^{\pm}] = [\overline{\tau}^{\pm}] = 0.$
This implies that $A(V)$ is generated by $[e], [f], [h] $ which satisfy commutation relations for $sl_2$. This proves assertion (i).
(Note that  Zhu's algebra of $N_{A_1} (k\Lambda_0)$  is
isomorphic to $U(sl_2)$.).

Now we shall prove relation (ii). By using Lemma \ref{lemma-c} we
get that in $A(V)$ we have
\bea  [G^{+}(-\tfrac{3}{2}) \overline{G}^{+} (-\tfrac{3}{2}) {\1}]
= -2 [e] [\omega]. \label{rel-z1} \eea
By definition of Zhu's algebra we have
\bea \label{rel-z2} [G^{+}(-\tfrac{3}{2}) \overline{G}^{+}
(-\tfrac{3}{2}) {\1}] + [G^{+}(-\tfrac{1}{2}) \overline{G}^{+}
(-\tfrac{3}{2}) {\1}] = 0. \eea
Since $[G^{+}(-\tfrac{1}{2}) \overline{G}^{+} (-\tfrac{3}{2})
{\1}] = [e(-2) \1] = - [e]$, relations (\ref{rel-z1}) and
(\ref{rel-z2}) gives that
$$- 2 [e][\omega] - [e] = 0,$$
which proves the proposition. \end{proof}

The previous proposition and Zhu's algebra theory imply that every irreducible $\tfrac{1}{2} \Zp$--graded $V$--module has the form
$ L_c ^{N=4} (U)$, where $U$ is a certain irreducible $U(sl_2)$--module annihilated by $[e] ([\omega]+1/2)$ and $ L_c ^{N=4} (U)$ is a $\tfrac{1}{2} {\Zp}$--graded $V$--module
$$  L_c ^{N=4} (U) = \bigoplus_{m \in \tfrac{1}{2} {\Zp} }  L_c ^{N=4} (U) (m), \quad  L_c ^{N=4} (U) (0) \cong U. $$

The classification of irreducible $V$--modules in the category $\mathcal{O}$ is then equivalent to the classification of
irreducible, highest weight modules for Zhu's algebra $A(V)$.

\begin{proposition} \label{class-zhu-1} Let $r \in {\Bbb C}$ and $U_r$ be the irreducible highest weight $sl_2$--module with highest weight
$r \omega_1$.
Then $U_0$ and $U_{-1}$ are the only irreducible, highest weight $A(V)$--modules. Therefore,
$ L_c ^{N=4} (U_0)$ and $L_c ^{N=4} (U_{-1})$ give a complete list of irreducible $V$--modules from the category
${\mathcal O}$.
\end{proposition}
\begin{proof}
Let $U=U_r$ be an irreducible, highest weight module for $A(V)$ and $ L_c ^{N=4} (U)$ the associated $V$--module.
Then
$$L(0) u =[\omega]. u = \frac{r (r+2)}{2} u \quad \forall u \in U. $$
This implies that
$$ [e] ([\omega]+ 1/2 ) U = 0 \implies U = 0 \quad \mbox{or}  \quad U = U_{-1}. $$
Clearly, $U_0$ is an $A(V)$--module since it is lowest component of $V$. $U_{-1}$ is also a module for  Zhu's algebra since the
lowest component of $M\otimes F$ is $( M \otimes F )(0) = U(sl_2) . e ^{-\delta} \cong U_{-1} $.
The proof follows.
\end{proof}

\begin{theorem} \label{simplicity}
$V$ is a simple vertex operator superalgebra, i.e., $V \cong L_c
^{N=4}$ with $c=-9$.
\end{theorem}
\begin{proof}Assume that $V$ is not simple. Then $V$ contains an graded
ideal $I \subset V$, $ 0 \ne I \ne V$ such that
$$ I = \bigoplus_{ n \in {\hf} {\Zp} } I(n+n_0), \quad L(0) \vert
I(r) \cong r \mbox{Id}, \quad I(n_0) \ne 0. $$
Since $I \ne V$, we have that $n_0 > 0$. Then the  lowest component
$I(n_0)  $ has to be a module for the Zhu's algebra $A(V)$.
Proposition \ref{zhu-str} then implies that either  $e (0)
I(n_0)=0$ or $L(0) \vert  I(n_0) =-\tfrac{1}{2} \mbox{Id}$. But
both cases lead to a contradiction. \end{proof}

We notice that $\delta(0)$ acts semi-simply on $V$ and it defines the ${\Z}$--graduation:
$$ V = \bigoplus_{\ell \in {\Z} } V ^{(\ell)}, \quad V^{(\ell) } = \{  v \in V \ \vert \ \delta(0) v = \ell v \}.$$
The simplicity of $V$ and the fact that
\bea && V^{(\ell_1)} \cdot V^{(\ell_2)} \subset V^{(\ell_1 + \ell_2)} \quad \forall \ell_1, \ell_2 \in {\Z} \label{grad-v} \eea
implies:

\begin{corollary}
$ V^{(0)}$ is a simple vertex algebra of central charge $c=-9$ and each $V^{(\ell)}$ is a simple $V^{(0)}$--module.
\end{corollary}

\begin{corollary} \label{characterization} Let $c=-9$. We have the following isomorphisms of vertex superalgebras:
$$  L_c ^{N=4} \cong ( M \otimes F
)^{int} \cong  \mbox{Ker} _{ M \otimes F}\widetilde{Q }.$$ \end{corollary}
\begin{proof}
Since all generators of $V$ belong to the vertex superalgebra $( M
\otimes F )^{int}$, we have that $V$ is its  subalgebra. As in the
previous theorem we conclude that $( M \otimes F )^{int}$ must be
simple as $V$--module. Similarly we prove that $V \cong \mbox{Ker} _{ M \otimes F}\widetilde{Q }$.
\end{proof}

Let:
$$ SC\Lambda (1) = (M \otimes F) ^{int} = V, \quad SC\Pi (1) = (M \otimes F) / SC\Lambda(1). $$

\begin{proposition}
\item[(1)] $SC \Pi (1)$ is an irreducible $U(\mathcal A)$--module
from the category $\mathcal{O}$.

\item[(2)] $SC\Lambda(1)$ and $SC\Pi(1)$ give all non-isomorphic
$V$--modules from the category $\mathcal{O}$.

\item[(3)] $M \otimes F$ is an indecomposable $V$--module.
\end{proposition}
\begin{proof}
Proof of assertion (i) uses the structure of Zhu's algebra $A(V)$ and it is similar to that of Theorem \ref{simplicity}.
Assertion (ii) follows from  Proposition \ref{class-zhu-1}, and (iii) from the fact that
$M \otimes F$ is not semi-simple $V$--module.
\end{proof}

\section{Irreducible $V$--modules outside of the category $\mathcal{O}$ }
\label{weight}
We have seen in the previous section that $V$ has only two irreducible modules in the category $\mathcal{O}$.
But modules from the
category $\mathcal{O}$ don't form a semisimple category. In this section we shall classify all irreducible,
$1/2 {\Zp}$--graded
 $V$--modules in the category of weight modules.

 We say that a $U(\mathcal A)$--module $M$ is a weight module if $L(0)$ and $h(0)$ act semisimply on $M$.

For every $(\mu,r) \in {\Bbb C}$ we define the following $U(sl_2)$--module $U_{\mu,r}$.
As a vector space $$ U_{\mu ,r} = {\mbox span}_{\C} \{ E_i, \ i \in {\Z} \},$$
the $sl_2$ action is defined by
\bea
 e E_i &=& E_{i-1} \nonumber \\
 h E_i & = & (- 2 r - 2 i + \mu) E_i \nonumber \\
 f E_i & = & -(r+i+1) (r+i-\mu)  E_{i+1}. \nonumber
\eea
Let $\Omega = e f + f e + \frac{1}{2} h^2 $ be the Casimir element of $U(sl_2)$. Then  $\Omega w = \frac{\mu  (\mu+2)}  {2} w$ for every $w \in U_{\mu ,r}$.

\begin{lemma}
Assume that $U$ is an irreducible, weight $A(V)$--module. Then $U$ is isomorphic to exactly one of the modules
from the list
$$ U _0, U_{-1}, U_{-1} ^{*}, U_{-1, r}  \quad (r \in {\C} \setminus \Z). $$
\end{lemma}

Now we want to show that $U_{-1,r}$ are modules for Zhu's algebra $A(V)$ and to construct the associated simple $V$--modules
$L^{N=4}_{c} (U_{-1,r})$.

First we notice that for every $\lambda \in {\C}$
$$V_{L + \lambda (\alpha + \beta) }:= V_L . e ^{\lambda (\alpha + \beta) } $$
is an $\sigma_{\lambda}$--twisted $V_{L}$--module, where $\sigma_{\lambda} := \exp[ \lambda (\alpha + \beta) (0)]$
is $V_L$--automorphism. This twisted module can be constructed as
$$ (V_{L + \lambda (\alpha + \beta) }, \widetilde{Y}(\cdot, z) ):=
(V_{L }, Y_{V_L} (\Delta (\lambda (\alpha + \beta)(-1){\1}), z)
\cdot, z) ).$$  Since $\sigma_{\lambda}
\equiv \mbox{Id}$ on $M \otimes F$ and $V$, we conclude that $V_{L
+ \lambda (\alpha + \beta) }$ is an untwisted $V$--module. Take
$\lambda =-r-1$ and consider $V$--submodule \bea &&  M {(r)}:={\C}
[ \beta + ({\Z} + \lambda)(\alpha + \beta)] \otimes M_{\alpha,
\beta} (1)
 \otimes F \label{opis} \eea
where $M_{\alpha, \beta} (1)$ is the Heisenberg subalgebra of $V_L$ generated by $\alpha(z)$ and $\beta (z)$.

\begin{theorem} \label{ired-relaxed}
Assume that $r \notin {\Z}$. Then $M {(r)}$ is an irreducible
$\tfrac{1}{2} {\Zp}$--graded $V$--module whose lowest component is
$$ M{(r)} (0) := U(sl_2). e ^{\beta - \delta  -(r+1) (\alpha + \beta) } \cong U_{-1,r}. $$
In particular,  $M {(r)} \cong L^{N=4}_{c} (U_{-1,r})$.
\end{theorem}
\begin{proof}
It is clear that the lowest component of $ M{(r)} $ is
$$ M{(r)} (0) = \mbox{span}_{\Bbb C} \{ E_i = e^{\beta - \delta - (r+1+ i) (\alpha+ \beta) }, \ i \in \Z \}, $$
and that $ L(0) E_i = -\frac{1}{2} E_i $, $i \in {\Z}$. By using explicit formula for $e,  h, f$ (\ref{def-e})-(\ref{def-f})
we see that $M{(r)} (0) \cong U_{-1,r}$ as $U(sl_2)$--modules and as $A(V)$--modules. Assume now that $M {(r)}$ is
not irreducible. Then it must contain a  graded submodule which is not generated by vectors from $M {(r)} (0)$.
By using information about Zhu's algebra $A(V)$, we conclude that then $M {(r)}$ should contain vector $w$ such that
$L(0) w= 0$, $e w = f w = h w =0$. But vectors of conformal weight $0$ are in the linear span  of
$$ \{ e^{\beta   - (r+1+ i) (\alpha+ \beta) },  e^{\beta   - (r+1+ i) (\alpha+ \beta) - 2 \delta  } \  \vert \  i \in \Z \}, $$
which can not generate a $V$--submodule with lowest conformal weight $0$. The proof follows.
\end{proof}

For any $V$--module $M$
such that $L(0)$ and $h(0)$ act semi-simply with finite-dimensional (common) eigenspaces we define
$\mbox{ch}_M (q,z) = \mbox{tr} q^{L(0)} z ^{h(0)}. $ By using (\ref{opis}) and properties of $\delta$--function one show the
following result.

\begin{proposition} We have
$$\mbox{ch}_{M(r)} (q,z) = z^{-2 r } \delta(z ^{2})  \prod_{n=1} ^{\infty} (1-q ^{n}) ^{-2}
\prod_{n=1} ^{\infty} ( 1+ q ^{n-3/2} z^{-1}) (1 + q^{n+1/2} z ).$$
\end{proposition}

\section{Logarithmic $V$--modules}

Now we shall apply the general method for constructing logarithmic modules developed in \cite{AdM-2009} to construct a series of logarithmic modules for the vertex superalgebra $V$.

First we notice that
\bea && \Pi(0) = {\Bbb C}[\Z (\alpha+\beta)] \otimes M_{\alpha, \beta} (1) \label{def-pi0} \eea
is a vertex subalgebra of $V_L$ and that for every $\lambda \in {\C}$
$$\Pi (\lambda) ={\Bbb C}[(\Z + \lambda )(\alpha+\beta)] \otimes M_{\alpha, \beta} (1) = \Pi(0). e^{\lambda (\alpha + \beta)}  $$
is its irreducible module.

\begin{remark}
This vertex algebra also appeared  in \cite{A-2007}, \cite{BDT} and \cite{efren}.
\end{remark}

Now define the following vertex superalgebra
$$ S\Pi (0) = \Pi (0) \otimes F \subset V_L$$
and its irreducible modules
$$ S\Pi (\lambda) = S\Pi (\lambda) \otimes F \subset V_{ L + \lambda(\alpha + \beta)}. $$

We note also
$$ V \subset M \otimes F \subset S\Pi(0). $$

Consider now the extended vertex superalgebra (in the sense of \cite{AdM-2009})
$$ \mathcal{SV} := S\Pi (0) \oplus S\Pi(-1/2)  $$
and its modules
$$ \mathcal{SV} (\lambda)  := S\Pi (\lambda)  \oplus S\Pi(\lambda-1/2)  .  $$
Let $$v = e ^{-\frac{1}{2} (\alpha + \beta ) + \delta} \in S\Pi(-1/2). $$
and recall that $\widetilde{Q} = \mbox{Res}_z Y(v,z)$ is a screening operator
which commutes with the action of $U(\mathcal A)$.

By using results of \cite{AdM-2009} we get:
\begin{theorem} For every $\lambda \in {\C}$
\bea && (\widetilde{\mathcal{SV}} (\lambda)  , \widetilde{Y}_{\widetilde{\mathcal{SV}} (\lambda) }(\cdot, z))
:=({\mathcal{SV}} (\lambda), Y_{{\mathcal{SV}} (\lambda)  }  (\Delta (v,z) \cdot,z )) \label{deformed} \eea
is a logarithmic $V$--module. The action of the Virasoro algebra is
$$ \widetilde{L(z)} = \sum_{n \in \Z} \widetilde{L(n)} z ^{-n-2} =
\widetilde{Y}_{\widetilde{\mathcal{SV}} (\lambda)}(\omega, z) = L(z) + z^{-1}
 Y_{{\mathcal{SV}} (\lambda)}(v, z). $$
 In particular,
 $$\widetilde{L(0)} = L(0) + \widetilde{Q}$$
 and $\widetilde{\mathcal{SV}} (\lambda)$ has $\widetilde{L(0)}$--nilpotent rank two.
\end{theorem}

Let us describe these logarithmic modules in more detail.

We have:
\bea &&  \Delta(v, z) e =  e, \nonumber \\
 &&
\Delta(v ,z) h = h, \nonumber  \\
&& \Delta(v, z) f =  f + \frac{1}{2} z^{-1} e^{-\frac{3}{2}(\alpha + \beta ) + \delta}, \nonumber \\
&& \Delta(v,z) \tau ^{\pm} = \tau ^{\pm}, \nonumber \\
&& \Delta(v,z) \overline{\tau}^{+} = \overline{\tau} ^{+} +  2 z^{-1} e^{\frac{1}{2}(\alpha + \beta ) }, \nonumber \\
&& \Delta(v,z) \overline{\tau} ^{-} = \overline{\tau} ^{-} - 2 z^{-1} D (e ^{-\frac{1}{2} (\alpha + \beta) })+
z ^{-2} e^{-\frac{1}{2} (\alpha + \beta)}.
\nonumber \eea
The formulas above and (\ref{deformed}) completely determine the   action of the $N=4$ superconformal algebra on logarithmic modules.

The following corollary shows that logarithmic $V$--modules appear in the extension of non-logarithmic (weight) $V$--modules:
\begin{corollary}
 The logarithmic $V$--module $\widetilde{\mathcal{SV}} (\lambda) $  appears in the extension
$$ 0 \rightarrow  S \Pi (\lambda -1/2) \rightarrow   \widetilde{\mathcal{SV}} (\lambda)  \rightarrow    S \Pi (\lambda )  \rightarrow 0 . $$
of non-logarithmic weak  $V$--modules $S \Pi (\lambda -1/2)$ and  $S \Pi (\lambda)$.
\end{corollary}

\section{ Twisted $V$--modules }

\label{twisted}

Let $ \mu \in {\R}$, $0 \le \mu  < 1$ and consider now the automorphism $g_{\mu} =e ^{ 2 \pi i \mu \delta (0) } $ of $M \otimes F$. One can easily see that $V$ is  $g_{\mu}$-invariant and that $g_{\mu}$--acts semisimply on $V$.

We have the following result on  the structure of  twisted Zhu's algebra $A_{g_{\mu} } (V)$.

\begin{proposition} \label{zhu-str-twisted} Assume that $\mu  + {\Z} \ne \tfrac{1}{2} + {\Z}. $
\item[(i)] Zhu's algebra $A_{g_{\mu} } (V)$ is isomorphic to a certain
quotient of $U(sl_2)$.

\item[(ii)] In Zhu's algebra $A_{g_{\mu} } (V)$ we have the following
relation:
$$ [e] ( [\omega] + \tfrac{  (1+ \mu) (1 - \mu ) }{2}) = 0.$$
\end{proposition}
\begin{proof}  Proof is similar to the untwisted case. Since  $A_{g_{\mu} } (V)$ is generated by
$$ [\tau ^{\pm} ], [{ \overline{\tau} } ^{\pm}], [e], [f], [h] ,  $$
and since
 $$ \tau ^{\pm}, \overline{\tau}^{\pm} \in O_{g_{\mu}} (V),$$ we conclude that
that $A_{g_{\mu}} (V)$ is  again generated by $[e], [f], [h] $ which satisfy commutation relations for $sl_2$. This proves assertion (i).

Definition of twisted Zhu's algebra $A_{g_{\mu}}(V)$ implies that
 \bea \label{rel-z3}  &&[G^{+}(-\tfrac{3}{2}) \overline{G}^{+}
(-\tfrac{3}{2}) {\1}] + ( 1 +  \mu)  [G^{+}(-\tfrac{1}{2}) \overline{G}^{+}
(-\tfrac{3}{2}) {\1}] + \nonumber \\ &&  { 1+ \mu \choose 2} [ G^{+}(\tfrac{1}{2}) \overline{G}^{+}
(-\tfrac{3}{2}) {\1}]  = 0. \eea

 Now assertion (ii) follows from Lemma \ref{lemma-c}. \end{proof}

First we will consider $g_{\mu}$--twisted $V$--module $M \otimes F ^{\mu}$. Recall that this module is a direct sum of Wakimoto modules for $\widehat{sl_2}$.

We have the following irreducibility result:

\begin{proposition}
Assume that $\mu \notin \tfrac{1}{2}{\Z}$. Then $M \otimes F^{\mu}$ is an irreducible $V$--module.
\end{proposition}
\begin{proof}
$M \otimes F ^{\mu}$ is $\tfrac{1}{2} {\Zp}$--graded and its lowest  component is generated by the vector $ 1 \otimes e ^{(\mu -1) \delta}$. This vector is a highest weight vector for $sl_2$ such that
$$h(0) 1 \otimes e ^{( \mu -1)\delta} =  (\mu-1) (1 \otimes e ^{( \mu-1) \delta}).$$
So, lowest component  is  an irreducible $sl_2$--module. Irreducibility of $M \otimes F^{\mu}$ follows easily by using structure of Zhu's algebra $A_{g_{\mu}} (V)$. Proof is similar to that of untwisted case.
\end{proof}

Next we consider a family of  $g_{\mu}$--twisted $V$--modules:

 $$ (M ^{\mu} (r), \widetilde{Y}_{ M ^{\mu} (r) } (\cdot, z) :=  (M ^{\mu} , Y_{ M  (r) } (\Delta (\mu \delta(-1) {\bf 1}) \cdot , z). $$
 Then $M ^{\mu} (r)$ is $g_{\mu}$--twisted $V$--module. $M ^{\mu} (r)$ can be realized as a submodule of $$ V_L. e^{\lambda (\alpha + \beta) + \mu \delta } $$
 where $\lambda=-r-1$. In fact:
 \bea &&  M ^{\mu} {(r)}:={\C}
[ \beta + ({\Z} + \lambda)(\alpha + \beta)] \otimes M_{\alpha,
\beta} (1)
 \otimes F^{\mu}, \label{opis-tw} \eea
where $F ^{\mu}$ is a  $g_{\mu}$--twisted $F$--module realized on $V_{\Z \delta}. e ^{\mu \delta}$.

\begin{theorem}
Assume that $r \notin {\Z}$, $ \mu \notin \tfrac{1}{2}{\Z}$ and $r - \mu \notin {\Z}$. Then $M ^{\mu} {(r)}$ is an irreducible $g_{\mu}$--twisted
$\tfrac{1}{2} {\Zp}$--graded $V$--module whose lowest component is
$$ M^{\mu} {(r)} (0) := U(sl_2). e ^{\beta  + (\mu -1)  \delta  -(r+1) (\alpha + \beta) } \cong U_{\mu -1,r}. $$
\end{theorem}
\begin{proof}
The proof uses the irreducibility of lowest component and the structure of twisted Zhu's algebra $A_{g_{\mu} }( V)$ from Proposition \ref{zhu-str-twisted}.
\end{proof}
Fix the following $\Z$--graduation on $M^{\mu} (s)$:
$$ M^{\mu} (s)= \bigoplus_{j \in {\Z} } ( M^{\mu} (s) ) ^j \quad (M^{\mu} (s)) ^j = \{ v \in M^{\mu} (s) \ \vert \delta(0) v = ( j + \mu -1) v \}. $$

\section{ Realization of the simple vertex operator algebra $L_{
A_2} (-\frac{3}{2}\Lambda_0) $ }

Let $\mathfrak{g}$ be the simple, complex Lie algebra of type $A_2$. The root system of $\mathfrak{g}$ is given by
$$ \Delta = \{ \varepsilon_i - \varepsilon_j \ \vert \ 1 \le i, j \le 3 \ i\ne j \} $$
with $\alpha_1 = \varepsilon_1-\varepsilon_2$, $\alpha_2 = \varepsilon_2 - \varepsilon_3 $ being simple roots. The highest
root is $\theta = \varepsilon_1 - \varepsilon_3$. We shall fix root vectors and coorots  as usual.
For any positive root $\alpha \in \Delta_+$ denote by $e_{\alpha}$ and
$f_{\alpha}$ the root vectors corresponding to $\alpha$ and $-\alpha$, respectively. Let $h_{\alpha}$ be the corresponding
coroot.

Let $\hat{\mathfrak g}$ be the affine Lie algebra of type $A_2 ^{(1)}$. Denote by $\Lambda_i$, $i=0,1,2$ its
fundamental weights.
Let $L_{A_2}( k \Lambda_0)$ be the simple vertex (operator) algebra of level $k$ associated to $\hat{\mathfrak g}$.

\bigskip

Now we shall prove that the simple affine vertex operator algebra  $L_{
A_2} (-\frac{3}{2}\Lambda_0) $ can be embedded into $V \otimes F_{-1}$, where $F_{-1}$ is the simple lattice
 vertex superalgebra $V_{\Z \varphi}$ associated to the lattice $\Z \varphi$ where
$\langle \varphi, \varphi \rangle =-1$. As a vector space,
 $F_{-1} =   M_{\varphi} (1) \otimes {\Bbb C}[\Z \varphi] $ where $M_{\varphi} (1)$
 is the Heisenberg vertex algebra generated
 by the field $\varphi(z)= \sum_{n \in {\Z} } \varphi(n) z ^{-n-1}$, and ${\Bbb C}[\Z \varphi] $
  is the group algebra with generator
 $ e^{\varphi}$. $F_{-1}$ admits the following $\Z$--graduation:
 $$F_{-1} = \bigoplus_{m \in \Z} F_{-1} ^{(m)}, \quad F_{-1} ^{(m)} = M_{\varphi} (1) \otimes e^{m \varphi}. $$

Define

\bea
e_{\theta} &:=& J^{+} \otimes {\bf 1}  =e  \otimes {\bf 1} \label{for-e13} \\
f_{\theta} &:=& J^{-} \otimes {\bf 1} = f \otimes  {\bf 1}  \label{for-f13} \\
e_{\alpha_1} &:=&  \frac{1}{\sqrt{2}}\  \tau^{+} \otimes e^{\varphi} \label{for-e12} \\
f_{\alpha_1} &:=& \frac{1}{\sqrt{2}} \  \overline{\tau} ^{-} \otimes e^{-\varphi} \label{for-f12} \\
e_{\alpha_2} &:=& \frac{1}{\sqrt{2}} \ \overline{\tau} ^{+} \otimes e^{-\varphi} \label{for-e23} \\
f_{\alpha_2} &:=& \frac{1}{\sqrt{2}}\ {\tau} ^{-} \otimes e^{\varphi} \label{for-f23} \\
h_{\alpha_1} &:=& (-\beta - \frac{3}{2} \varphi +\frac{1}{2} \delta) (-1) \label{for-h1} \\
h_{\alpha_2} &:=& (-\beta + \frac{3}{2} \varphi +\frac{1}{2} \delta) (-1) \label{for-h2}.
\eea

\begin{lemma}
Formulas (\ref{for-e13})-(\ref{for-h2}) define a vertex algebra homomorphism:
$$ \Phi : N_{A_2} (-\frac{3}{2} \Lambda_0) \rightarrow V \otimes F. $$
\end{lemma}
\begin{proof}
Since $ \tau^+, \overline{\tau}^-$ generate $N=2$ superconformal vertex algebra with $c=-9$, then $e_{\alpha_1}, f_{\alpha_1}, h_{\alpha_1}$ must  generate affine vertex algebra $A_1 ^{(1)} $ at level $k=-3/2$, i.e., $L_{A_1}(-\frac{3}{2} \Lambda_0)$ (see \cite{A-IMRN} and \cite{FST}). Similarly $e_{\alpha_2}, f_{\alpha_2}, h_{\alpha_2}$ also generate a copy of $L_{A_1}(-\frac{3}{2} \Lambda_0)$. By using defining relations for the $N=4$ superconformal algebra we have:
\bea
&& e_{\alpha_1} (0) f_{\alpha_2} = e_{\alpha_2} (0)  f_{\alpha_1} = 0,  \nonumber \\
&& e_{\alpha_1} (0) e_{\alpha_2} = \frac{1}{2}   (G^{+} (1/2)  \overline{ G} ^+ (-3/2) {\bf 1} )  \otimes {\bf 1} = e \otimes {\bf 1}= e_{\theta}, \nonumber \\
&& f_{\alpha_1} (0) f_{\alpha_2} =  \frac{1}{2}   (\overline{G} ^{-} (1/2)  G^- (-3/2) {\bf 1} )  \otimes {\bf 1} = -f  \otimes {\bf 1}= -f _{\theta}, \nonumber \\
&& e(0) f_{\alpha_1}=   \frac{1}{\sqrt{2}} ( e(0) \overline{\tau}^-) \otimes e^{-\varphi}= -\frac{1}{\sqrt{2}} \overline{\tau} ^+ \otimes e^{-\varphi} =- e_{\alpha_2}, \nonumber  \\
&& e(0) f_{\alpha_2}=   \frac{1}{\sqrt{2}} ( e(0) {\tau}^-) \otimes e^{\varphi}= \frac{1}{\sqrt{2}} {\tau} ^+ \otimes e^{\varphi} = e_{\alpha_1},  \nonumber \\
&& f(0) e _{\alpha_1}=   \frac{1}{\sqrt{2}} ( f(0) \tau ^+  ) \otimes e^{\varphi}= \frac{1}{\sqrt{2}} \overline{\tau} ^-  \otimes e^{\varphi} = f_{\alpha_2}, \nonumber  \\
&& f(0) e_{\alpha_2}=   \frac{1}{\sqrt{2}} ( f(0)  \overline{\tau}^-) \otimes e^{-\varphi}= -\frac{1}{\sqrt{2}} \overline{ {\tau} } ^-  \otimes e^{-\varphi} =- f_{\alpha_1}.  \nonumber
\eea
Now assertion follows from  the commutator formula for vertex algebras.
\end{proof}

Denote by $\pi_s$, $s \in \Z$, the automorphism of $U(\hg)$ uniquely determined by
$$ \pi_s( e_{\alpha_1} (n) ) = e_{\alpha_1} (n+s),   \pi_s( e_{\alpha_2} (n) ) = e_{\alpha_2} (n-s) $$
$$ \pi_s( f_{\alpha_1} (n) ) = f_{\alpha_1} (n-s),   \pi_s( f_{\alpha_2} (n) ) = f_{\alpha_2} (n+s) $$
$$ \pi_s (h_{\alpha_1} (n) ) =h_{\alpha_1} (n) + s k \delta_{n,0},   \pi_s (h_{\alpha_2} (n) ) =h_{\alpha_2} (n) - s k \delta_{n,0}. $$

In our explicit realization operator $e ^{s \varphi}$ is such automorphism. We have:

\begin{theorem}
The subalgebra of $V \otimes F_{-1}$ generated by vectors (\ref{for-e13})-(\ref{for-h2}) is isomorphic
to the vertex operator algebra
$L_{A_2} (-\frac{3}{2}\Lambda_0) $. Moreover, $V \otimes F_{-1}$ is a simple current extension of $L_{
A_2} (-\frac{3}{2}\Lambda_0) $ and
$$ V \otimes F_{-1} =  \bigoplus_{s \in {\Z} } \pi_s (L_{
A_2} (-\frac{3}{2}\Lambda_0) ) $$
where $\pi_s (L_{
A_2} (-\frac{3}{2}\Lambda_0) ) = e ^{s \varphi}  L_{
A_2} (-\frac{3}{2}\Lambda_0)$.
\end{theorem}
\begin{proof}
First we notice that $(\delta+ \varphi )(0)$ acts semi-simply on the simple vertex superalgebra
$V \otimes F_{-1}$ and defines the following ${\Z}$--graduation
$$ V \otimes F_{-1} = \bigoplus_{\ell} (V\otimes F_{-1} ) ^{(\ell)},   $$
where
$$(V\otimes F_{-1} ) ^{(\ell)} =\{ v \in V\otimes F_{-1} \ \vert \ (\delta + \varphi ) (0) v = \ell v \}. $$
Then $(V\otimes F_{-1} ) ^{(0)}$ is a simple vertex algebra and each $(V\otimes F_{-1} ) ^{(\ell)}$ is a simple
$(V\otimes F_{-1} ) ^{(0)}$--module.

 For the proof of simplicity, it suffices to prove that $(V\otimes F_{-1} ) ^{(0)}$ is generated by vectors (\ref{for-e13})-(\ref{for-h2}).
 This can be proved by using relations
 \bea
 \tau ^+ \otimes {\bf 1}&=& \sqrt{2} e_{\alpha_1} (-2)  . ( {\bf 1} \otimes e^{-\varphi} ), \nonumber \\
 \tau ^- \otimes {\bf 1}&=& \sqrt{2} f_{\alpha_2} (-2)  . ( {\bf 1} \otimes e^{-\varphi} ), \nonumber \\
 \overline{\tau} ^+ \otimes {\bf 1}&=& \sqrt{2} e_{\alpha_2} (-2)  . ( {\bf 1} \otimes e^{\varphi} ), \nonumber \\
 \overline{\tau} ^- \otimes {\bf 1}&=& \sqrt{2} f_{\alpha_1} (-2)  . ( {\bf 1} \otimes e^{\varphi} ). \nonumber
 \eea
 and  analogous proof to that of Theorem 6.1 from \cite{A-2007}.
\end{proof}

\begin{corollary}
We have the following inclusions of vertex algebras
$$ L_{
A_2} (-\frac{3}{2}\Lambda_0) \subset M \otimes \widetilde{\Pi} (0) \subset \Pi(0) \otimes \widetilde{\Pi} (0) $$
where $\Pi(0)$ is defined by (\ref{def-pi0}) and
$$\widetilde{\Pi} (0)={\C} [ \Z (\delta + \varphi)] \otimes M_{\delta, \varphi} (1). $$
\end{corollary}

Let $\mu \in {\R}$. Consider now $h_{\mu}:= e^{ 2 \pi i  \mu \varphi(0) }$ twisted $F_{-1}$--module
$$ F_{-1}^{\mu} := V_{\Z \varphi} . e ^{\mu \varphi}. $$
As a $M_{\varphi}(1)$--module:
$$ F_{-1} ^{\mu} = \bigoplus_{ j \in {\Z} } F_{-1} ^{\mu + j}, \quad F_{-1} ^{\mu + j} = M_{\varphi} (1). e ^{ ( j + \mu) \varphi}. $$
Let  $U$ be any $g_{\mu}:= e^{ 2 \pi i \mu \delta(0) }$--twisted $V$--module. Since $g_{\mu} h_{\mu} = e ^{ 2 \pi i \mu (\delta + \varphi) (0) }$ acts trivially on $L_{A_2} (-\frac{3}{2} \Lambda_0) = (V \otimes F_{-1} ) ^0 $,  we conclude that $U \otimes F_{-1} ^{\mu}$ is an untwisted $ L_{A_2} ( - \frac{3}{2} \Lambda_0)$--module.
The proof of the following result is analogous to that of Theorem 6.2 from \cite{A-2007}.

\begin{theorem} \label{osnovni-ir-A2}
Assume that $U$ is a $g_{\mu}$--twisted  $V$--module such that $U$ admits the
following ${\Z}$--graduation
\bea && U = \bigoplus_{j \in {\Z} } U^{(j)}, \quad V^{(i)} .
U^{(j)} \subset U ^{(i+j)} . \label{uvjet-grad} \eea
Then
$$ U \otimes F^{\mu} _{-1} = \bigoplus_{s \in {\Z} } {\mathcal L} _{s}(U), \quad  \mbox{where} \ \    {\mathcal L} _{s} (U)
 := \bigoplus_{i \in {\Z} } U^{i} \otimes F_{-1} ^{- s+i + \mu },$$
  is an  $ L_{A_2} (-\frac{3}{2}\Lambda_0)$--module.

If $U$ is irreducible $g_{\mu}$--twisted $V$--module, then for every $s \in {\Z}$  $ {\mathcal
L}_{s}(U)$ is an irreducible (untwisted)  $L_{A_2} (-\frac{3}{2}\Lambda_0)$--module.
\end{theorem}

Irreducible modules in the category $\mathcal{O}$ were classified in \cite{P}. This vertex operator algebra is rational in the category $\mathcal{O}$ (cf. \cite{AM}, \cite{Ar}). Here we have free field realization of these modules:
\begin{corollary} All irreducible  $L_{A_2} (-\frac{3}{2} \Lambda_0)$--modules in the category $\mathcal{O}$ can be realized as follows:
\bea &&  L_{A_2} (-\frac{3}{2} \Lambda_0)\cong U(\hg). {\bf 1}, \quad L_{A_2} (-\frac{1}{2} \Lambda_0- \frac{1}{2} \Lambda_1 - \frac{1}{2}  \Lambda_2)\cong U(\hg). e^{- \delta}, \nonumber \\
&&  L_{A_2} (-\frac{3}{2} \Lambda_1) \cong  U(\hg). e ^{ - \beta + \tfrac{1}{2} \delta - \tfrac{1}{2} \varphi }, \quad   L_{A_2} (-\frac{3}{2} \Lambda_2) \cong  U(\hg). e ^{ - \beta + \tfrac{1}{2} \delta + \tfrac{1}{2} \varphi }. \nonumber
\eea
\end{corollary}

Using our realization one easily see that the category $\mathcal{O}$ is not closed under fusion.

Now we will apply Theorem \ref{osnovni-ir-A2} to  irreducible untwisted  $V$--modules from Section \ref{weight}  and twisted $V$--modules constructed in  Section \ref{twisted}.

\begin{corollary} \label{ireducibilni-2}  Assume that $0 < \mu < 1$, $\mu \ne 1/2$.
\item[(1)] For every $r \in {\C} \setminus {\Z}$ and every $s \in
\Z$, $\mathcal{L}_s (M(r))$ is an irreducible $L_{A_2}
(-\frac{3}{2}\Lambda_0)$--module.

\item[(2)]  For  every $s \in
\Z$, $\mathcal{L} _s (M \otimes F^{\mu} )$ is an irreducible $L_{A_2}
(-\frac{3}{2}\Lambda_0)$--module.

\item[(3)] For every $r \in {\C} \setminus {\Z}$ such that $r  - \mu  \notin {\Z}$ and every $s \in
\Z$, $\mathcal{L}_s (M ^{\mu} (r))$ is an irreducible $L_{A_2}
(-\frac{3}{2}\Lambda_0)$--module.

\end{corollary}

 By construction, the Sugawara Virasoro vector in $L_{A_2} (-\frac{3}{2}\Lambda_0) $ is
\bea && \omega_{A_2}  = \omega -\frac{1}{2} \varphi(-1) ^2 {\bf 1}. \label{sug-a2}  \eea
So in Zhu's algebra we have $[\omega_{A_2}] = [\omega] - 1/2 [\varphi (-1) {\bf 1}]^2 $. Using  this relation and $[e_{\alpha_1} \circ e_{\alpha_2} ] = 0$, we get the following relation in  Zhu's algebra $A(L_{A_2} (-\frac{3}{2} \Lambda_0) )$:
\bea \label{rel-A2} [e] ([\omega_{A_2}] + 1/2 )= 0. \eea
(See also \cite{P} for analysis of Zhu's algebra of this vertex algebra). So we have analogous relations in Zhu's algebras  $A(V)$  and $A(L_{A_2} (-\frac{3}{2} \Lambda_0))$.

Among irreducible modules constructed in Corollary  \ref{ireducibilni-2} we have a family of $\Zp$--graded modules realized on
 the irreducible $\Pi(0) \otimes \widetilde{\Pi}
(0)$--module:
$$\mathcal{L}_0 (M ^{\mu} (r)  ) \cong \Pi(0) \otimes \widetilde{\Pi} (0) .  e^{ \beta - \delta - (r+1) (\alpha + \beta) + \mu (\delta + \varphi) }. $$
From technical reasons our proof  uses assumption that $\mu $ is real. But relation  (\ref{rel-A2}) enables us to present an  alternative proof of irreducibility which works  for complex parameters $\mu$.

\begin{corollary} \label{ireducibilni-3} Assume that $(r, \mu ) \in {\C} ^2$ such that  $r \notin {\Z}$,  $\mu \notin  (\tfrac{1}{2} + {\Z} )$ and $r - \mu \notin {\Z}$. Then  on the  $\Pi(0) \otimes \widetilde{\Pi}
(0)$--module
$$ \mathcal{L}_0 (M ^{\mu} (r)  ) = \Pi(0) \otimes \widetilde{\Pi} (0) .  e^{ \beta - \delta - (r+1) (\alpha + \beta) + \mu (\delta + \varphi) } $$
exists the structure of an irreducible $\Zp$--graded   $L_{A_2}
(-\frac{3}{2}\Lambda_0)$--module.

\end{corollary}
\begin{proof} Let $E_{i, j}:= e ^{ \beta -\delta + (-r-1 -i ) (\alpha + \beta) + (\mu + j)( \delta+ \varphi) }$. Then the lowest component of $\mathcal{L}_0 (M ^{\mu} (r)  )$ is spanned by $\{ E_{i,j} \ \vert \ i, j \in {\Z} \}$. By using a direct calculation we get
\bea
e_{\alpha_1} (0) E_{i , j } & = & \frac{1}{\sqrt{2}} E_{i, j +1},  \nonumber \\
e_{\alpha_2} (0) E_{i, j } & = & \frac{1}{\sqrt{2}} (1- 2  \mu- 2 j )  E_{i-1, j -1 },  \nonumber \\
f_{\alpha_1} (0) E_{i , j} & = & -\frac{1}{\sqrt{2}} (1- 2 \mu -2 j)  (r+ i - \mu  -j +1 ) E_{i, j -1}, \nonumber \\
f_{\alpha_2} (0) E_{i, j } & = & \frac{1}{\sqrt{2}} ( r+ i + 1)  E_{i+1, j +1 }.  \nonumber
\eea
So the lowest component of $ \mathcal{L}_0 (M ^{\mu} (r)  )$ is an irreducible $A_2$--module. Now irreducibility follows easily from relation  (\ref{rel-A2}) in Zhu's algebra $A(L_{A_2} (-\frac{3}{2} \Lambda_0) )$ and similar arguments as in  the proof of Theorem \ref{ired-relaxed}.
\end{proof}

\begin{remark}
 Modules in Corollary \ref{ireducibilni-3} don't belong to the category $\mathcal{O}$. Their lowest components are irreducible $sl_3$--modules having all $1$--dimensional weight spaces (cf. \cite{BL}).
\end{remark}

Finally we have a family of logarithmic modules:

\begin{corollary}
For every $\lambda \in {\C}$
$ \widetilde{\mathcal{SV} } (\lambda) \otimes F_{-1} $ is an logarithmic $L_{A_2}
(-\frac{3}{2}\Lambda_0)$--module.
\end{corollary}

\section{ Connection with $\mathcal{W}_{A_2} (p)$--algebras: $p=2$ }

Let $\h = {\C} h_{\alpha_1} + {\C}  h_{\alpha_2}$ be the Cartan subalgebra of $\g =sl_3$.  Then the Heisenberg subalgebra of $L_{A_2}( -\frac{3}{2} \Lambda_0) $ generated by $\h$  is isomorphic to $M_{\h} (1)$. In this section we shall study the parafermion vertex operator algebra $ K(sl_3,-\tfrac{3}{2})= \mbox{Com} (M_{\h }(1), L_{A_2} (-\frac{3}{2}\Lambda_0) )$ and relate it  with vertex algebras  $\mathcal{W}_{A_2} (p)$ appearing in logarithmic conformal field theory.

We consider the lattice
$$ \sqrt{p} A_2
 = {\Z} \gamma_1 + {\Z} \gamma_2, \quad \langle \gamma_1, \gamma_1\rangle = \langle \gamma_2, \gamma_2 \rangle = 2 p,
 \ \langle \gamma_1, \gamma_2 \rangle = - p. $$

Let $M_{\gamma_1, \gamma_2} (1)$ be the standard Heisenberg vertex subalgebra of the lattice vertex algebra
 $V_{ \sqrt{p} A_2}$ generated by the Heisenberg fields $\gamma_1 (z)$ and $\gamma_2 (z)$.

 The vertex algebra $\mathcal{W}_{A_2} (p)$ is defined (cf. \cite{AdM-peking}, \cite{S1} ) as a subalgebra of the lattice vertex algebra $V_{ \sqrt{p} A_2}$ realized as
 $$ \mathcal{W}_{A_2} (p) = \mbox{Ker} _{ V_{ \sqrt{p} A_2} } e^{-\gamma_1 /p } _0 \bigcap
 \mbox{Ker} _{ V_{ \sqrt{p} A_2} } e^{-\gamma_2 /p } _0 . $$
We also have its subalgebra:
$$ \overline{M_{\gamma_1, \gamma_2} (1) } = \mbox{Ker} _{ M_{\gamma_1, \gamma_2} (1) } e^{-\gamma_1 /p } _0 \bigcap
 \mbox{Ker} _{ M_{\gamma_1, \gamma_2} (1) } e^{-\gamma_2 /p } _0   =  W_{A_2} (p) \bigcap M_{\gamma_1, \gamma_2} (1). $$
$\mathcal{W}_{A_2} (p)$ and  $\overline{M_{\gamma_1, \gamma_2} (1) }$ have a  vertex subalgebra isomorphic to the simple
$\mathcal{W}(2,3)$--algebra  with central charge $c_p=2-24 \frac{(p-1) ^{2}}{p}$.

It is expected that $\mathcal{W}_{A_2} (p)$ is a $C_2$--cofinite vertex algebra for $p \ge 2$ and that  it is  a completely reducible
$\mathcal{W}(2,3) \times sl_3$--module. Then as a vertex algebra $\mathcal{W}_{A_2} (p)$ is  strongly generated by
$\mathcal{W}(2,3)$ generators and
by $sl_3 . e ^{-\gamma_1 - \gamma_2 }$, so by $8$ primary fields for the $\mathcal{W}(2,3)$--algebra.

Note that $\mathcal{W}_{A_2} (p) $ is a generalization of the triplet vertex algebra $\mathcal{W}(p)$ and
$ \overline{M_{\gamma_1, \gamma_2} (1) }$ is a generalization of the singlet vertex subalgebra of $\mathcal{W}(p)$.

\bigskip

In the case $p=2$ we shall embed all vertex algebras defined above in the lattice vertex algebra $V_L$. Define
$$\gamma_1 = -2 \alpha, \quad \gamma_2 = \alpha + \beta - 2 \delta. $$
Then $\Z \gamma_1 + \Z \gamma_2 \cong \sqrt{2} A_2$.

\begin{lemma} We have:
$$
 K(sl_3,-\tfrac{3}{2}) \subset  \overline{M_{\gamma_1, \gamma_2} (1) } .$$
\end{lemma}
\begin{proof}
First we notice that:
$$ L(-\frac{3}{2} \Lambda_0) \subset \Pi(0) \otimes \widetilde{\Pi}(0) \cong {\C}[ {\Z} (\alpha + \beta) + {\Z} (\delta + \varphi) ] \otimes
M_{\h  }(1) \otimes M_{\gamma_1, \gamma_2} (1). $$
This implies that
$$ \mbox{Com} (M_{\h  }(1), L_{A_2} (-\frac{3}{2}\Lambda_0) ) \subset  M_{\gamma_1, \gamma_2} (1),  $$
where $ M_{\gamma_1, \gamma_2} (1)$ is identified with the subalgebra $ 1 \otimes 1 \otimes   M_{\gamma_1, \gamma_2} (1)$.
The assertion of the lemma follows from fact that $e ^{ \alpha}_0$ and $\widetilde{Q}$ act trivially on  $L(-\frac{3}{2} \Lambda_0) $.
\end{proof}

By using the following   realization of the Weyl vertex algebra $$ M= \mbox{Ker} _{\Pi (0)} e^{\alpha}_0$$
(see \cite{A-2007}, \cite{efren} for details) we conclude that
$$ \mbox{Ker} _{ M_{\gamma_1, \gamma_2} (1)} e^{-\frac{1}{2} \gamma_1}_0 = M \otimes F \bigcap M_{\gamma_1, \gamma_2} (1).$$

Since $ V = \mbox{Ker}_{M \otimes F} e^{-\frac{1}{2} \gamma_2}_0$, we get
$$ \overline{M_{\gamma_1, \gamma_2} (1) } = V \cap M_{\gamma_1, \gamma_2} (1) \subset V ^{ (0) }  \subset L_{A_2} (-\frac{3}{2}\Lambda_0) ). $$

So we have proved:
\begin{theorem} \label{coset-a2} We have:
$$ \overline{M_{\gamma_1, \gamma_2} (1) } \cong
 K(sl_3,-\tfrac{3}{2}) .$$
\end{theorem}

\begin{remark}
 In the case $g = sl_2$ we have $K(sl_2, -\frac{1}{2}) \cong W(2,3)_{-2}$  (cf. \cite{R}, \cite{W}) and  $K(sl_2, -\frac{4}{3}) \cong W(2,5)_{-7}$ (cf. \cite{A-JPAA}) where
  $W(2,3)_{-2}$ and  $ W(2,5)_{-7}$  are singlet vertex algebras realized as subalgebras of  triplet vertex algebras $\mathcal{W}(p)$ for $p=2,3$. Theorem \ref{coset-a2} shows that for admissible vertex algebras associated to $\widehat{sl_3}$ at level $k=-3/2$ we have  interpretation of the coset $K(sl_3,k)$ in the framework of vertex algebras which are  higher rank generalizations of the triplet vertex algebras.
\end{remark}

\section{Generalizations and future work}

We  shall discuss some possible generalizations of the present work.

Bx Corollary \ref{characterization} we have that the simple $N=4$ vertex superalgebra with $c=-9$ $L_{c} ^{N=4}$ (denoted  here  by $V$ )  is isomorphic to
$\mbox{Ker}_{  M \otimes F } \widetilde{Q}. $
  We have seen that in the case $k+2 =\frac{1}{p}$
affine vertex algebra $L_{A_1} (k \Lambda_0)$  is realized  inside of the  generalized vertex superalgebra $M \otimes F_{p/2}$ where
$F_{p/2} = V_{ {\Z} \frac{p}{2} \delta }$ and we have introduced the following (generalized) vertex algebras:
$$ {\mathcal V} ^{(p)} = \mbox{Ker}_{  M \otimes F_{p/2} } \widetilde{Q}. $$ We conjecture  that in this case ${\mathcal V}  ^{(p)}$ is strongly generated by
generators $e,f,h$ of  $L_{A_1} (k \Lambda_0)$ and
\bea  && \tau_{(p)} ^+  = e^{ \frac{p}{2} \delta}, \nonumber \\
&&  \overline{\tau}_{(p)} ^+  = Q  e^{ \frac{p}{2} \delta},  \nonumber \\
&&     {\tau}_{(p)} ^- = f(0) e^{ \frac{p}{2} \delta}, \nonumber \\
&&       \overline{\tau}_{(p)} ^- = -f(0)  Q e^{ \frac{p}{2} \delta}. \nonumber \eea

\begin{remark}
Drinfeld-Sokolov functor sends $L_{A_1} (k \Lambda_0)$ to the simple Virasoro vertex algebra $L(c_{1,p}, 0)$ with central charge of $(1,p)$--models. This suggests that  ${\mathcal V}^{(p)}$ is mapped to the doublet vertex algebra $\mathcal{A}(p)$ and that a $\Z_2$ orbifold of ${\mathcal V} ^{(p)}$ is naturally mapped to the triplet vertex algebra $\mathcal{W}(p)$. These  vertex algebras  can be considered as logarithmic extensions of $L_{A_1} (k  \Lambda_0)$. It is expected that their representation-categories are connected  with Nichols algebras studied by A. M.  Semikhatov and I. Yu Tipunin \cite{ST}.
\end{remark}

Based on the case $p=2$ we expect that the following conjecture holds:

\begin{conjecture}
For every $p \ge 3$ ${\mathcal V} ^{(p)}$ has finitely many irreducible modules in the category $\mathcal{O}$. There exists non-semisimple $ {\mathcal V} ^{(p)}$--modules from the category $\mathcal{O}$.
\end{conjecture}

In \cite{A-JPAA}, we studied relations between admissible affine vertex algebra $L_{A_1} (-\frac{4}{3} \Lambda_0 )$ and vertex algebras associated to $(1,3)$--models for the Virasoro algebra (singlet, doublet and triplet vertex algebras). Some constructions of  \cite{A-JPAA} were  generalized (mostly conjecturally) in \cite{CRW} where the authors found a connection between  $(1,p)$ models and Feigin-Semikhatov $W$-algebras $W_2 ^{(n)}$.

 In our case the realization of the admissible simple affine vertex operator algebra $L_{A_2} (- \frac{3}{2} \Lambda_0)$ also admits a natural  generalization. Let $F_{-p/2}$ denotes the generalized lattice vertex algebra associated to the lattice $\Z  (\frac{p}{2} \varphi)$ such that
 $$ \la \varphi , \varphi \ra = - \frac{2}{p}. $$

Let   $\mathcal{R} ^{ (p) }$ by the subalgebra of $ {\mathcal V} ^{(p)} \otimes F_{-p/2}$ generated by
$ x= x(-1) {\bf 1} \otimes 1$, $x \in \{ e, f, h \} $, $1 \otimes \varphi(-1) {\bf 1}$  and

\bea
e_{\alpha_1,p} &:=&  \frac{1}{\sqrt{2}}\  \tau^{+}_{(p)}  \otimes e^{ \tfrac{p}{2}\varphi} \label{for-e12-p} \\
f_{\alpha_1,p }  &:=& \frac{1}{\sqrt{2}} \  \overline{\tau} ^{-}_{(p)}  \otimes e^{- \tfrac{p}{2} \varphi} \label{for-f12-p} \\
e_{\alpha_2,p }  &:=& \frac{1}{\sqrt{2}} \ \overline{\tau} ^{+} _{(p)} \otimes e^{- \tfrac{p}{2} \varphi} \label{for-e23-p} \\
f_{\alpha_2,p }&:=& \frac{1}{\sqrt{2}}\ {\tau} ^{-} _{(p) } \otimes e^{\tfrac{p}{2} \varphi} \label{for-f23-p}
\eea
Clearly, $\mathcal{R} ^{(2)} \cong L_{A_2} (-\frac{3}{2} \Lambda_0). $ In general, $\mathcal{R} ^{(p)}$ is an extension of
$$ L_{A_1} ( (\tfrac{1}{p}   - 2) \Lambda_0) \otimes M_{\varphi} (1) $$ by $4$ fields of conformal weight  $p/2$.

We believe that $\mathcal{R} ^{(p)}$  for $p \ge 3$ is also part of a series of generically existing vertex (super)algebras.

In order to present some evidence for this statement, we consider the universal  affine $\mathcal{W}$ algebras $\mathcal{W} ^k ({\g}, f_{\theta})$ associated with
$({\g}, f_{\theta})$ where $\g$ is a simple Lie algebra and  $f_{\theta}$ is a root vector associated to the  lowest root -$\theta$. Let  $\mathcal{W} _k ({\g}, f_{\theta})$ be its simple quotient.
Let ${\g} = sl_4$. Then by Proposition 4.1. of \cite{KWR}, $\mathcal{W}^k (sl_4, f_{\theta})$ is generated by $4$ four vectors  of conformal weight  one which generate affine vertex algebra associated to  $ \widehat{gl_2}$ at level $k +1$, Virasoro vector  and four even vectors of conformal weight $3/2$. By using concepts from \cite{AP} (slightly generalized for $\mathcal{W}$--algebras) one can easily show that there is a conformal embedding of $L_{gl(2) } (-\frac{5}{3} \Lambda_0)$ into simple vertex algebra $\mathcal{W}_k (sl_4, f_{\theta})$. Therefore $\mathcal{W}_k (sl_4, f_{\theta})$ for $k=-8/3$ is also an extension of $$ L_{A_1} (-\frac{5}{3} \Lambda_0) \otimes M_{\varphi} (1) $$
by four fields of conformal weight $3/2$.

This supports  the following conjecture:

\begin{conjecture} We have:
$$  \mathcal{R} ^{(3)} \cong \mathcal{W}_k (sl_4, f_{\theta})  \quad \mbox{for} \ k = - \frac{8}{3}. $$
\end{conjecture}

It is clear that these vertex algebras also admit logarithmic representations and have interesting fusion rules. We plan to investigate the representation  theory of these vertex algebras in our forthcoming papers.

\end{document}